\documentclass[12pt]{article}
\usepackage{indentfirst, latexsym, bm,amssymb}

\begin{document}

\newtheorem{theorem}{Theorem}[section]
\newtheorem{corollary}[theorem]{Corollary}
\newtheorem{definition}[theorem]{Definition}
\newtheorem{conjecture}[theorem]{Conjecture}
\newtheorem{question}[theorem]{Question}
\newtheorem{lemma}[theorem]{Lemma}
\newtheorem{proposition}[theorem]{Proposition}
\newtheorem{example}[theorem]{Example}
\newenvironment{proof}{\noindent {\bf
Proof.}}{\rule{3mm}{3mm}\par\medskip}
\newcommand{\remark}{\medskip\par\noindent {\bf Remark.~~}}
\newcommand{\pp}{{\it p.}}
\newcommand{\de}{\em}

\newcommand{\JEC}{{\it Europ. J. Combinatorics},  }
\newcommand{\JCTB}{{\it J. Combin. Theory Ser. B.}, }
\newcommand{\JCT}{{\it J. Combin. Theory}, }
\newcommand{\JGT}{{\it J. Graph Theory}, }
\newcommand{\ComHung}{{\it Combinatorica}, }
\newcommand{\DM}{{\it Discrete Math.}, }
\newcommand{\ARS}{{\it Ars Combin.}, }
\newcommand{\SIAMDM}{{\it SIAM J. Discrete Math.}, }
\newcommand{\SIAMADM}{{\it SIAM J. Algebraic Discrete Methods}, }
\newcommand{\SIAMC}{{\it SIAM J. Comput.}, }
\newcommand{\ConAMS}{{\it Contemp. Math. AMS}, }
\newcommand{\TransAMS}{{\it Trans. Amer. Math. Soc.}, }
\newcommand{\AnDM}{{\it Ann. Discrete Math.}, }
\newcommand{\NBS}{{\it J. Res. Nat. Bur. Standards} {\rm B}, }
\newcommand{\ConNum}{{\it Congr. Numer.}, }
\newcommand{\CJM}{{\it Canad. J. Math.}, }
\newcommand{\JLMS}{{\it J. London Math. Soc.}, }
\newcommand{\PLMS}{{\it Proc. London Math. Soc.}, }
\newcommand{\PAMS}{{\it Proc. Amer. Math. Soc.}, }
\newcommand{\JCMCC}{{\it J. Combin. Math. Combin. Comput.}, }
\newcommand{\GC}{{\it Graphs Combin.}, }

\title{The Wiener and Terminal Wiener indices of trees \thanks{
This work is supported by National Natural Science Foundation of
China (No:11271256).  }}
\author{ Ya-Hong  Chen$^{1,2}$,  Xiao-Dong Zhang$^1$\thanks{Corresponding  author ({\it E-mail address:}
xiaodong@sjtu.edu.cn)}
\\
{\small $^1$Department of Mathematics, and  MOE-LSC,}\\
{\small Shanghai Jiao Tong University} \\
{\small  800 Dongchuan road, Shanghai, 200240,  P.R. China}\\
 {\small $^2$Department of Mathematics},
{\small Lishui University} \\
{\small  Lishui, Zhejiang 323000, PR China}}
%\date{}
\maketitle
 \begin{abstract}Heydari  \cite{heydari2013}  presented  very nice
 formulae
 for the Wiener and terminal Wiener indices of generalized Bethe trees. It
 is pity that there are some errors for the formulae. In this paper,
 we correct these errors and characterize all trees with the minimum
 terminal Wiener index among all the trees of order $n$ and with maximum degree $\Delta$.

   \end{abstract}


\vskip 0.5cm

\section{Introduction}
  There are many molecular structure descriptors until now.
The Wiener index is one of the most widely known topological
descriptors, which has been much studied in both mathematical and
chemical literatures (for example, see\cite{dobrynin2001,
gutman-k-2, Gutman1997-k}). Through this paper, we only consider
finite, simple and undirected graphs. Let $G=(V(G), E(G))$ be a
simple connected graph of order $n$ with vertex set $V(G)$ and edge
set $E(G)$. The distance between vertices $v_i$ and $v_j$ is the
minimun  number of edges between $v_i$ and $v_j$  and denoted by
$d_G(v_i,v_j)$ (or for short $d(v_i,v_j)$). The {\it Wiener index}
of a connected graph $G$ is defined as the sum of distances between
all pairs of vertices:
\begin{eqnarray*}
 W(G)=\sum\limits_{{v_i,v_j}\subseteq V(G)}d(v_i,v_j)=\frac{1}{2}\sum\limits_{v\in
 V(G)}d_G(v),
\end{eqnarray*}
where $d_G(v)$ denotes the distance of a vertex $v$. For trees,
Wiener \cite{wiener1947} gave a very useful formula for the Wiener
index: \begin{eqnarray}\label{equ1} W(G)=\sum\limits_{e\in
T}n_1(e)n_2(e),
\end{eqnarray}
where $n_1(e)$ and $n_2(e)$ are the number of vertices of two
components of $T-e$.
  Recently, Smolenski et al. \cite{smolenskii2009} made use of terminal distance matrices to encode molecular structures.
Based on these applications, Gutman, Furtula and Petrovi\'{c}
\cite{gutman2009} proposed the concept of {\it terminal Wiener index}, which is defined as the sum of distances between all pairs of pendent vertices of trees:
\begin{eqnarray*}
 TW(T)=\sum\limits_{1\leq i<j\leq k} d_T(v_i,v_j),
\end{eqnarray*}
where $d_T(v_i,v_j)$  is the distance of two pendent vertices $v_i$
and $v_j$. Gutman  gave a similar formula for the terminal Wiener
index of trees
\begin{eqnarray}\label{eqn2}
 TW(T)=\sum\limits_{e\in T}p_1(e)p_2(e),
\end{eqnarray}
where $n_1(e)$ and $n_2(e)$ are the number of vertices of two
components of $T-e$.
  For more information on the Wiener and terminal Wiener indices, the readers may refer to
  the recent papers \cite{heydari2010,schmuck2012,szekely2011} and the references cited therein.

{\it A generalized Bethe tree} (see \cite{heydari2013}) is a rooted tree whose vertices at the same
 level have equal degrees. We agree that the root vertex is at level 1 and $T$ has $k$ levels, and denote the class of
generalized Bethe trees of $k$ levels by $\mathcal{B}_k$. {\it A
Bethe tree} $B_{k,d}$ is a rooted tree of $k$ levels in which the
root vertex has degree $d$, the vertices at  level $j(2\leq j\leq
k-1)$  have $d+1$ and the vertices at level $k$ are the pendent
vertices. A {\it regular dendrimer tree} $T_{k,d}$ is a generalized
Bethe tree of $k+1$ levels with each nonpendent vertex  having degree $d$. So a
regular dendrimer tree belongs to  $\mathcal{B}_{k+1}$.

  The rest of the paper is organized as follows. In section 2, we present
  some formulae for the Wiener index of generalized Bethe trees, which correct
  the errors of \cite{heydari2013}. In section 3,
  a formula for the  terminal Wiener indices of trees is obtained. With the formula,
  the terminal Wiener index of a general Bethe tree is presented, which
   corrects the errors of \cite{heydari2013}. In section 4, the trees with the minimum
terminal Wiener index among all the trees of order $n$ and with
maximum degree $\Delta$ are characterized.

\section{\large\bf{Wiener index of generalized Bethe trees}}

  Let $T_1,T_2,\cdots,T_m(m\geq2)$ be trees with disjoint vertex sets and
  orders $n_1,n_2,\cdots,$ $n_m$.
Let $w_i\in V(T_i)$ be the rooted vertex of $T_i$ for $i=1,2,\cdots,m$.
 A  tree $T$ on more than
two vertices can be regarded as being obtained by joining a new
vertex $w$ to each of the vertices $w_1,w_2,\cdots,w_m$. Canfield,
Robinson and Rouvray \cite{canfield1985} elaborated a recursive
approach for calculation of the Wiener index of a general tree.
Dobrynin Entringer and Gutman \cite{dobrynin2001} state the result
as the following theorem.

\begin{theorem}
\label{theorem1.1}
(\cite{dobrynin2001})
  Let $T$ be a  tree on $n\geq3$ vertices, whose structure is specified above. Then
   $$W(T)=\sum\limits_{i=1}^m[W(T_i)+(n-n_i)d_{T_i}(w_i)-n_i^2]+n(n-1), $$
  where $d_{T_i}(w_i)$ is the sum of distances between $w_i$ and all other vertices of $T_i$ for $1\leq i\leq m$.
\end{theorem}

  Since a generalized Bethe tree is the very special tree
  whose vertices have the same degree at the same level,
Heydari \cite{heydari2013} presented a formula for the Wiener index
of generalized Bethe trees. The result can be stated as follows:

\begin{theorem}
\label{theorem1.2}
(\cite{heydari2013})
  Let $B_{k+1}$ be a generalized Bethe tree of $k+1$ levels. If $d_1$ denotes the degree of rooted vertex
  and $d_i+1$ denotes degree of the vertices on $i$th level of  $B_{k+1}$ for $1<i<k$, then
  the Wiener index of $B_{k+1}$ is computed as follows:
    $$W(B_{k+1})=\sum\limits_{i=1}^k(n_{i+1}-1)m_i(n-m_i), $$
  where $n_{i+1}$ is the number of vertices  on the $(i+1)$th level of $B_{k+1}$ and
  $m_i$ is the number of all children vertices lying on one side of edge
  where adjacent a vertex on the $i$th level to
   another vertex on $(i+1)$th level of $B_{k+1}$ for $1\leq i\leq k$.
\end{theorem}

\setlength{\unitlength}{1mm}
\begin{picture}(120, 40)
\centering
\multiput(60,30)(0,0){1}{\circle*{1}}
\multiput(40,20)(40,0){2}{\circle*{1}}
\multiput(30,10)(10,0){3}{\circle*{1}}
\multiput(70,10)(10,0){3}{\circle*{1}}
  \thicklines
\put(60,30){\line(-2,-1){20}}
\put(60,30){\line(2,-1){20}}
\put(40,20){\line(-1,-1){10}}
\put(40,20){\line(0,-1){10}}
\put(40,20){\line(1,-1){10}}
\put(80,20){\line(-1,-1){10}}
\put(80,20){\line(0,-1){10}}
\put(80,20){\line(1,-1){10}}
\put(50,0){\small{Figure 1  ~~$B_3$ }}
\end{picture}

Unfortunately, this result is not correct. For example[see figure
1]: $B_3$ is a generalized Bethe tree with 9 vertices. It is easy to
see that $k=2,n_2=2,n_3=6,m_1=4,m_2=1$. Using the formula as above,
we have $W(B_3)=60$. But actually, the Wiener index of $W(B_3)=88$.
In here, we present a correct formula for the Wiener index of a
generalized Bethe tree.
\begin{theorem}
\label{theorem1.3}
  Let $B_{k+1}$ be a generalized Bethe tree of $k+1$ levels. If $d_1$ denotes the degree of rooted vertex
  and $d_i+1$ denotes the degree of vertices on $i$th level of  $B_{k+1}$ for $1<i\leq k$,
  then
  \begin{eqnarray}\label{eqn3}
 W(B_{k+1})=\sum\limits_{i=1}^k n_{i+1}m_i(n-m_i),
\end{eqnarray}
where $n_{i+1}=d_1d_2\cdots d_i$ and
  $m_i=1+\sum\limits_{j=i+1}^k\prod_{r=i+1}^id_r$ for $1\leq i\leq k$.
\end{theorem}

\begin{proof} Let $n_i$ be the number of vertices on the $i$th level of $B_{k+1}$.
 Thus $n_1=1$
and $n_i=d_1d_2d_3\cdots d_{i-1}$ for $i=2,3,\cdots,k+1$. Denote by
$|V(B_{k+1})|=n$. Then
$$n=n_1+n_2+\cdot\cdot\cdot+n_{k+1}=1+\sum\limits_{i=1}^k\prod_{j=1}^id_j.$$
Suppose that $u$ on the $i$th level of $B_{k+1}$ for $1\leq i\leq k$
is the parent of $v$. So all of the children of vertex $v$ are lying
one side of edge $e=uv$. Denote by $m_i$ the number of those
vertices of the tree. Then
$$m_i=1+d_{i+1}+d_{i+1}d_{i+2}+
 \cdot\cdot\cdot +d_{i+1}d_{i+2}\cdots d_k=1+\sum\limits_{j=i+1}^k\prod_{r=i+1}^id_r$$
for $1\leq i\leq k$. Obviously, $m_k=1$. Hence the number of
vertices where lying two sides of $e$ are equal to $n_1(e)=m_i$ and
$n_2(e)=n-m_i$, respectively. Since the number of edges of $B_{k+1}$
where adjacent a vertex on the $i$th level to another vertex on
$(i+1)$th level of $B_{k+1}$ is equal to $n_{i+1}$. By using (1), we
have
 $$W(B_{k+1})=\sum\limits_{e=E(B_{k+1})}n_1(e)n_2(e)=\sum\limits_{i=1}^k n_{i+1}m_i(n-m_i) $$
The proof is completed.
\end{proof}

  By using the correct formula (\ref{eqn3}), it is easy to check that $W(B_3)=88$.
  Obviously, the dendrimer tree $T_{k,d}$ is one of the special generalized Bethe trees.

\begin{corollary}
\label{corollary1.1}
  Let $T_{k,d}$  be a dendrimer tree of $k+1$ levels where degree of the nonpendent vertices is equal to $d$. Then
  the Wiener index of $T_{k,d}$  is computed as follows:
  \begin{equation}\label{equ4}
W(T_{k,d})=\frac{d}{(d-2)^3}[(d-1)^{2k}(kd^2-2(k+1)d+1)+2d(d-1)^k-1].
\end{equation}
\end{corollary}
\begin{proof} Since  the degree of nonpendent vertices of $T_{k,d}$ is equal to $d$,
we have $n_1=1$, $n_i=d(d-1)^{i-2}$ for $2\leq i\leq k+1$,
$n=1+\frac{d((d-1)^k-1)}{d-2}$ and
$m_i=1+\frac{(d-1)((d-1)^{k-i}-1)}{d-2}.$
 By (\ref{eqn3}), we have
\begin{eqnarray*}
W(T_{k,d})&=&\sum\limits_{i=1}^kd(d-1)^{i-1}[1+\frac{(d-1)[(d-1)^{k-i}-1]}{d-2}][\frac{d((d-1)^k-1)}{d-2}\\
&&-\frac{(d-1)[(d-1)^{k-i}-1]}{d-2}]\\
&=&\sum\limits_{i=1}^kd(d-1)^{i-1}\frac{[d-2+(d-1)[(d-1)^{k-i}-1]}{d-2}\times\\
&&\frac{[d(d-1)^k-(d-1)^{k-i+1}-1]}{d-2}\\
&=&\sum\limits_{i=1}^kd(d-1)^{i-1}\frac{[(d-1)^{k-i+1}-1]}{d-2}\times\frac{[d(d-1)^k-(d-1)^{k-i+1}-1]}{d-2}\\
&=&\frac{d}{(d-2)^2}\sum\limits_{i=1}^k[(d-1)^k-(d-1)^{i-1}][d(d-1)^k-(d-1)^{k-i+1}-1]\\
&=&\frac{d}{(d-2)^2}\sum\limits_{i=1}^k[d(d-1)^{2k}-(d-1)^{2k-i+1}-d(d-1)^{k+i-1}+(d-1)^{i-1}]\\
&=&\frac{d}{(d-2)^2}[kd(d-1)^{2k}-\sum\limits_{i=1}^k
(d-1)^{2k-i+1}-d\sum\limits_{i=1}^k (d-1)^{k+i-1}+\sum\limits_{i=1}^k (d-1)^{i-1}]\\
&=&\frac{d}{(d-2)^2}[kd(d-1)^{2k}-\frac{(d-1)^{k+1}[(d-1)^k-1]}{d-2}-\\
&&\frac{d(d-1)^{k}[(d-1)^k-1]}{d-2}+\frac{(d-1)^k-1}{d-2}]\\
&=&\frac{d}{(d-2)^3}[(d-1)^{2k}(kd^2-2(k+1)d+1)+2d(d-1)^k-1]
\end{eqnarray*}
The proof is completed.
\end{proof}

\begin{corollary}
\label{corollary1.2}
 The Wiener index of  a Bethe tree $B_{k,d}$  is computed as follows:
    $$W(B_{k,d})=\frac{d^k}{(d-1)^3}[(k-1)(d-1)(d^k+1)-2d(d^{k-1}-1)]$$
\end{corollary}
\begin{proof} Since degree of the nonpendent vertices of $B_{k,d}$ is equal to $d+1$ except the rooted vertex whose degree is $d$, we have
$n_1=1$, $n_{i+1}=d^{i}$ for $1\leq i\leq{k-1}$,
$n=\frac{d^k-1}{d-1}$ and $m_i=\frac{d^{k-i}-1}{d-1}$. By
(\ref{eqn3}), we can get
\begin{eqnarray*}
W(B_{k,d})&=&\sum\limits_{i=1}^{k-1}d^{i}\frac{d^{k-i}-1}{d-1}(\frac{d^k-1}{d-1}-\frac{d^{k-i}-1}{d-1})\\
%&=&\sum\limits_{i=1}^{k-1}\frac{d^k-d^i}{d-1}\frac{d^k(1-d^{-i})}{d-1}\\
&=&\frac{d^k}{(d-1)^2}\sum\limits_{i=1}^{k-1}(d^k-d^{k-i}-d^i+1)\\
&=&\frac{d^k}{(d-1)^2}[(k-1)(d^k+1)-\sum\limits_{i=1}^{k-1}d^{k-i}-\sum\limits_{i=1}^{k-1}d^i
] \\
&=&\frac{d^k}{(d-1)^2}[(k-1)(d^k+1)-\frac{2 d(d^{k-1}-1)}{d-1}]\\
&=&\frac{d^k}{(d-1)^3}[(k-1)(d-1)(d^k+1)-2d(d^{k-1}-1)].
\end{eqnarray*}
The proof is completed.
\end{proof}

\section{\large\bf{Terminal Wiener index of trees}}

 In this section, we consider the terminal Wiener index of trees.
 For a tree $T$ with order $n\geq3$ with rooted $w$,
 let $T_1,T_2,\cdots,T_m(m\geq2)$ be components of $T-w$
with orders $n_1,n_2,\cdots,$ $n_m$, respectively, where $w_i$ is
adjacent to the vertex $w$ in $T$ and is the rooted vertex in $T_i$.
 Let $l$ be the number of pendent vertices in $T$ and $l_i$ $(1\leq
i \leq m)$ be the number of pendent vertices in $T_i$. Clearly,
$l_1+l_2+\cdots+l_m=l$. We present a formula for computing the
terminal Wiener index of a tree by the terminal Wiener index of
subtrees.

\begin{theorem}
\label{theorem3.1}
  Let $T$ be a  tree with order $n\geq3$, whose structure is described as above. Then
   \begin{equation}\label{th31}
   TW(T)=\sum\limits_{i=1}^m[TW(T_i)+(l-l_i)d'_{T_i}(w_i)-l_i^2]+l^2,
   \end{equation}
  where $d'_{T_i}(w_i)$ is the sum of distances between $w_i$ and
   all other pendent vertices of $T_i$ for $1\leq i\leq m$.
\end{theorem}
 \begin{proof} Let $x_{ij}(1\leq j\leq l_i)$ be the pendent vertex in $T_i(1\leq i\leq m)$. Then

 \begin{eqnarray*}
TW(T)&=&\sum\limits_{i=1}^mTW(T_i)+\sum_{k=1}^{l_2}\sum_{h=1}^{l_1}d(x_{1h},x_{2k})+\sum_{k=1}^{l_3}\sum_{h=1}^{l_1}d(x_{1h},x_{3k})+\cdots+\\
&&\sum_{k=1}^{l_m}\sum_{h=1}^{l_1}d(x_{1h},x_{mk})+\sum_{k=1}^{l_3}\sum_{h=1}^{l_2}d(x_{2h},x_{3k})+\cdots+\sum_{k=1}^{l_m}\sum_{h=1}^{l_2}d(x_{2h},x_{mk})\\
&&+\cdots+\sum_{k=1}^{l_m}\sum_{h=1}^{l_{m-1}}d(x_{(m-1)h},x_{mk})\\
&=&\sum\limits_{i=1}^mTW(T_i)+\sum_{i=2}^m\sum_{k=1}^{l_i}\sum_{h=1}^{l_1}d(x_{1h},x_{ik})+\sum_{i=3}^m\sum_{k=1}^{l_i}\sum_{h=1}^{l_2}d(x_{2h},x_{ik})\\
&&+\cdots+\sum_{i=m-1}^m\sum_{k=1}^{l_i}\sum_{h=1}^{l_{m-2}}d(x_{(m-2)h},x_{ik})+\sum_{k=1}^{l_m}\sum_{h=1}^{l_{m-1}}d(x_{(m-1)h},x_{mk})
\end{eqnarray*}
Since the sum of distances between pendent vertices in each $T_i$ and $T_j$ can be calculated, i.e
\begin{eqnarray*}
\sum_{k=1}^{l_j}\sum_{h=1}^{l_i}d(x_{ih},x_{jk})=l_j d'_{T_i}(w_i)+l_i d'_{T_j}(w_j)+2l_il_j
\end{eqnarray*}
and $l^2=(l_1+l_2+\cdots+l_m)^2=\sum\limits_{i=1}^ml_i^2+2\sum\limits_{1\leq i<j\leq m}l_il_j$, then we have

\begin{eqnarray*}
TW(T)&=&\sum\limits_{i=1}^mTW(T_i)+(l_2+l_3+\cdots+l_m)d'_{T_1}(w_1)+(l_1+l_3+\cdots+l_m)d'_{T_2}(w_2)\\
&&+\cdots+(l_1+l_2+\cdots+l_{m-1})d'_{T_m}(w_m)+2\sum\limits_{1\leq i<j\leq m}l_il_j\\
&=&\sum\limits_{i=1}^mTW(T_i)+(l-l_1)d'_{T_1}(w_1)+(l-l_2)d'_{T_2}(w_2)+\cdots\\
&&+(l-l_m)d'_{T_m}(w_m)+l^2-\sum\limits_{i=1}^ml_i^2\\
&=&\sum\limits_{i=1}^m[TW(T_i)+(l-l_i)d'_{T_i}(w_i)-l_i^2]+l^2.
\end{eqnarray*}
We finish the proof.\end{proof}

  With (\ref{th31}),  the formulae for the terminal Wiener index of generalized Bethe
  trees, Bethe trees $B_{k,d}$ and $T_{k,d}$ are obtained, which
  correct the errors of  \cite{heydari2013}.

\begin{theorem}
\label{theorem3.2}
  Let $B_{k+1}$ be a generalized Bethe tree of $k+1$ levels. Then
  \begin{eqnarray}\label{eqn6}
 TW(B_{k+1})=\prod\limits_{i=1}^kd_i
 \times(k\prod\limits_{i=1}^kd_i-1-\sum\limits_{i=1}^{k-1}\prod\limits_{j=1}^i
 d_{k-j+1}).
\end{eqnarray}
\end{theorem}

\begin{proof} The pendent vertices of the generalized Bethe tree $B_{k+1}$ are located on the final level of the tree.
Let $n'$ be the number of pendent vertices of $B_{k+1}$, then
$n'=d_1d_2\cdots d_k$. Suppose that $e=uv$ is an edge of $B_{k+1}$,
and $u$ is the parent of $v$ on the $i$th level of $B_{k+1}$ for
$1\leq i\leq k$. Let $m'_i$ and $m''_i$ be the number of pendent
vertices of $B_{k+1}$, lying on the two sides of $e$ ,then
$m'_i=d_{i+1}d_{i+2}\cdots d_k$ and $m''_i=n'-d_{i+1}d_{i+2}\cdots
d_k$ for $1\leq i\leq k-1$. Obviously, $m'_k=1$ and $m''_k=n'-1$.
Since we have mentioned in Theorem~\ref{theorem1.2} that $n_{i+1}$
which stands for the number of edges where adjacent a vertex on the
$i$th level to another vertex on the $(i+1)$th level of $B_{k+1}$ is
equal to $d_1d_2\cdots d_i$ for $1\leq i\leq k+1$,  by using
(\ref{eqn2}), we have

\begin{eqnarray*}
TW(B_{k+1})&=&\sum\limits_{e\in E(B_{k+1})}p_1(e)p_2(e)\\
&=&\sum\limits_{i=1}^{k-1}n_{i+1}m'_i m''_i+n'm'_k m''_k\\
&=&\sum\limits_{i=1}^{k-1}n_{i+1}d_{i+1}d_{i+2}\cdots d_k(n'-d_{i+1}d_{i+2}\cdots d_k)+n'(n'-1)\\
&=&\sum\limits_{i=1}^{k-1}d_1d_2\cdots d_id_{i+1}d_{i+2}\cdots d_k(d_1d_2\cdots d_k-d_{i+1}d_{i+2}\cdots d_k)+\\
&&d_1d_2\cdots d_k(d_1d_2\cdots d_k-1)\\
&=&(k-1)(\prod\limits_{i=1}^{k}d_i)^{2}-\prod\limits_{i=1}^{k}d_i \sum\limits_{i=1}^{k-1}d_{i+1}d_{i+2}\cdots d_k
+\prod\limits_{i=1}^{k}d_i(\prod\limits_{i=1}^{k}d_i-1)\\
&=&\prod\limits_{i=1}^{k}d_i\times[(k-1)(\prod\limits_{i=1}^{k}d_i)+\prod\limits_{i=1}^{k}d_i-1-\sum\limits_{i=1}^{k-1}d_{i+1}d_{i+2}\cdots d_k]\\
&=&\prod\limits_{i=1}^kd_i \times(k\prod\limits_{i=1}^kd_i-1-\sum\limits_{i=1}^{k-1}\prod\limits_{j=1}^i d_{k-j+1})
\end{eqnarray*}
The proof is completed.
\end{proof}

  From Theorem~\ref{theorem3.2}, we can get the terminal Wiener index of $T_{k,d}$ .
\begin{corollary}
\label{corollary1.1}
  Let $T_{k,d}$  be a dendrimer tree of $k+1$ levels where degree of the nonpendent vertices is equal to $d$. Then
  the terminal Wiener index of $T_{k,d}$  is computed as follows:
    $$TW(T_{k,d})=d(d-1)^{k-1}[kd(d-1)^{k-1}+\frac{1-(d-1)^{k}}{d-2}]$$
\end{corollary}
\begin{proof} Since degree of the nonpendent vertices of $T_{k,d}$ is equal to $d$, it is easy to see that
$d_1$ is equal to $d$ and $n_i$ is equal to $d-1$ for $2\leq i\leq
k$. Then
$$n'=\prod\limits_{i=1}^{k}d_i=d(d-1)^{k-1}$$ and
\begin{eqnarray*}
\sum\limits_{i=1}^{k-1}\prod\limits_{j=1}^i d_{k-j+1}&=&\sum\limits_{i=1}^{k-1}d_{i+1}d_{i+2}\cdots d_k\\
&=&\sum\limits_{i=1}^{k-1}(d-1)^{k-i}\\
&=&\frac{(d-1)[(d-1)^{k-1}-1]}{d-2}.
\end{eqnarray*}
By using (\ref{eqn6}), we have
\begin{eqnarray*}
TW(T_{k,d})&=&d(d-1)^{k-1}[kd(d-1)^{k-1}-1-\frac{(d-1)[(d-1)^{k-1}-1]}{d-2}]\\
&=&d(d-1)^{k-1}[kd(d-1)^{k-1}+\frac{1-(d-1)^{k}}{d-2}].
\end{eqnarray*}
The proof is completed.
\end{proof}

\begin{corollary}
\label{corollary1.1}
  Let $B_{k,d}$  be a Bethe tree of $k$ levels. Then
    $$TW(B_{k,d})=\frac{d^{k-1}}{d-1}[d^{k-1}(kd-k-d)+1].$$
\end{corollary}
\begin{proof} Since $B_{k,d}$ is a Bethe tree  of level $k$, we replace $k$ in formula (\ref{eqn6})
 by $k-1$.
According to the definition of the Bethe tree $B_{k,d}$, it is easy
to see that  $\prod\limits_{i=1}^{k-1}d_i=d^{k-1}$ and
\begin{eqnarray*}
\sum\limits_{i=1}^{k-2}\prod\limits_{j=1}^i d_{k-j}&=&\sum\limits_{i=1}^{k-2}d_{i+1}d_{i+2}\cdots d_{k-1}\\
&=&\sum\limits_{i=1}^{k-2}d^{k-i-1}\\
&=&\frac{d[d^{k-2}-1]}{d-1}.
\end{eqnarray*}
By using (\ref{eqn6}), we have
\begin{eqnarray*}
TW(B_{k,d})&=&d^{k-1}[(k-1)d^{k-1}-1-\frac{d(d^{k-2}-1)}{d-1}]\\
&=&d^{k-1}[(k-1)d^{k-1}-\frac{d^{k-1}-1}{d-1}]\\
&=&\frac{d^{k-1}}{d-1}[(k-1)(d-1)d^{k-1}-d^{k-1}+1]\\
&=&\frac{d^{k-1}}{d-1}[d^{k-1}(kd-k-d)+1].
\end{eqnarray*}
The proof is completed.
\end{proof}

\section{\large\bf{Terminal Wiener index versus maximum degree in trees}}

  Let $\mathcal{T}(n,\Delta)$ denote the set of all the
  trees of order $n$ and with maximum degree $\Delta$.
In this section, we will characterize the trees with the minimum
 terminal Wiener index in $\mathcal{T}(n,\Delta)$.

\setlength{\unitlength}{1mm}
\begin{picture}(120, 40)
\centering\thicklines
\put(20,20){\circle{15}}
\put(75,20){\circle{15}}
\put(27,20){\circle*{1}}
\put(30,25){\circle*{1}}
\put(35,25){\circle*{1}}
\put(41.5,25){\circle*{1}}
\put(46.5,25){\circle*{1}}
\put(82,20){\circle*{1}}
\put(85,25){\circle*{1}}
\put(90,25){\circle*{1}}
\put(96.5,25){\circle*{1}}

\put(30,15){\circle*{1}}
\put(35,15){\circle*{1}}

\put(46.5,15){\circle*{1}}
\put(51.5,15){\circle*{1}}

\put(85,15){\circle*{1}}
\put(90,15){\circle*{1}}

\put(101.5,15){\circle*{1}}
\put(106.5,15){\circle*{1}}
\put(111.5,15){\circle*{1}}

\put(27,20){\line(3,5){3}}
\put(27,20){\line(3,-5){3}}
\put(30,25){\line(1,0){5}}
\put(35,25){\line(1,0){2}}
\put(37.5,25){\line(1,0){0.5}}
\put(38.5,25){\line(1,0){0.5}}
\put(39.5,25){\line(1,0){0.5}}
\put(39.5,25){\line(1,0){2}}
\put(41.5,25){\line(1,0){5}}

\put(30,15){\line(1,0){5}}
\put(35,15){\line(1,0){3}}
\put(39,15){\line(1,0){0.5}}
\put(40.5,15){\line(1,0){0.5}}
\put(42,15){\line(1,0){0.5}}

\put(43.5,15){\line(1,0){3}}

\put(46.5,15){\line(1,0){5}}

\put(82,20){\line(3,5){3}}
\put(82,20){\line(3,-5){3}}
\put(85,25){\line(1,0){5}}
\put(90,25){\line(1,0){2}}
\put(92.5,25){\line(1,0){0.5}}
\put(93.5,25){\line(1,0){0.5}}
\put(94.5,25){\line(1,0){2}}

\put(85,15){\line(1,0){5}}
\put(90,15){\line(1,0){3}}
\put(94,15){\line(1,0){0.5}}
\put(95.5,15){\line(1,0){0.5}}
\put(97,15){\line(1,0){0.5}}
\put(98.5,15){\line(1,0){3}}
\put(101.5,15){\line(1,0){5}}
\put(106.5,15){\line(1,0){5}}

\put(16,19){$R$}
\put(24,19){$r$}
\put(28,26){$u_1$}
\put(33,26){$u_2$}
\put(38,26){$u_{a-1}$}
\put(46.5,26){$u_a$}
\put(28,12){$v_1$}
\put(33,12){$v_2$}
\put(43,12){$v_{b-1}$}
\put(52,12){$v_{b}$}

\put(71,19){$R$}
\put(79,19){$r$}
\put(83,26){$u_1$}
\put(88,26){$u_2$}
\put(94.5,26){$u_{a-1}$}

\put(83,12){$v_1$} \put(88,12){$v_2$} \put(98,12){$v_{b-1}$}
\put(105.5,12){$v_{b}$} \put(111,12){$u_a$}
\put(30,5){\small{$T_{a,b}$ }} \put(90,5){\small{$T_{a-1,b+1}$ }}
\put(41,0){\small{Figure 2~~~$T_{a,b}$ and $T_{a-1,b+1}$}}
\end{picture}

In order to prove our main result,  we introduce a tree
transformation. Let $T_{a,b}$ and $T_{a-1,b+1}$ be the trees
depicted in Figure 2, where $b>a\ge 1$ are integers and $R$ is a
rooted tree with root $r$ and at least two vertices. Gutman,
Vuki\u{c}evi\'{c} and Petrovi\'{c} proved

\begin{lemma}(\cite{gutman2004}) Let $b>a\ge 1$. Then
\label{lemma1.1}
 $$W(T_{a,b})<W(T_{a-1,b+1}).$$
\end{lemma}
However, the above result is not true for terminal Wiener index. In fact,
\begin{lemma}
\label{lemma1.2}
 If $b>a>1$, then
 $$TW(T_{a,b})=TW(T_{a-1,b+1}).$$
\end{lemma}
\begin{proof} Suppose that there are $k$ pendent vertices of $R$ which are
labelled by $x_1,x_2,\cdots ,x_k$. Then
\begin{eqnarray*}
TW(T_{a,b})&=&\sum\limits_{1\leq x_i<x_j\leq k}d(x_i,x_j)+\sum\limits_{i=1}^{k}d(x_a,x_i)+\sum\limits_{i=1}^{k}d(v_b,x_i)+a+b\\
&=&\sum\limits_{1\leq x_i<x_j\leq k}d(x_i,x_j)+2\sum\limits_{i=1}^{k}d(r,x_i)+(a+b)k+a+b\\
\end{eqnarray*}
and
\begin{eqnarray*}
TW(T_{a-1,b+1})&=&\sum\limits_{1\leq x_i<x_j\leq k}d(x_i,x_j)+\sum\limits_{i=1}^{k}d(u_{a-1},x_i)+\sum\limits_{i=1}^{k}d(x_a,x_i)+a+b\\
&=&\sum\limits_{1\leq x_i<x_j\leq k}d(x_i,x_j)+2\sum\limits_{i=1}^{k}d(r,x_i)+(a+b)k+a+b.\\
\end{eqnarray*}
It is easy to see that $TW(T_{a,b})=TW(T_{a-1,b+1})$. The proof is
completed.
\end{proof}

 \begin{lemma}
\label{lemma1.3}
 If $b>a=1$, then
 $$TW(T_{a,b})>TW(T_{a-1,b+1}).$$
\end{lemma}
\begin{proof} Suppose that there are $k$ pendent vertices of $R$ which are
 labelled by $x_1,x_2,\cdots ,x_k$. Then
\begin{eqnarray*}
TW(T_{a,b})
&=&\sum\limits_{1\leq x_i<x_j\leq k}d(x_i,x_j)+2\sum\limits_{i=1}^{k}d(r,x_i)+(b+1)k+b+1\\
\end{eqnarray*}
and
\begin{eqnarray*}
TW(T_{a-1,b+1})
&=&\sum\limits_{1\leq x_i<x_j\leq k}d(x_i,x_j)+\sum\limits_{i=1}^{k}d(r,x_i)+(b+1)k.\\
\end{eqnarray*}
So $TW(T_{a,b})-TW(T_{a-1,b+1})=\sum\limits_{i=1}^{k}d(r,x_i)+b+1>0$. The proof is completed.
\end{proof}

  A tree is said to be {\it starlike} of degree $k$ if exactly one of its vertices has
degree greater than two, and the degree is equal to $k\geq3$.

\begin{theorem}
\label{theorem1.3}
If $T$ is a tree in $\mathcal{T}(n,\Delta)(\Delta\geq3)$, then
$$TW(T)\ge (n-1)(\Delta-1)$$
with equality if and only if $T$ is starlike of order $n$ with the maximum degree $\Delta$.
\end{theorem}
\begin{proof}
Since $T\in\mathcal{T}(n,\Delta)$, there exists at least one vertex labelled by $v$
such that $d(v)=\Delta$. So there are $\Delta$ branches of $T-v$.
If $T$ is not a starlike tree, there exist some branches of $T$ at $v$ that are not paths. Hence by Lemmas~\ref{lemma1.2} and~\ref{lemma1.3}, there exist a starlike tree $T_1$ of order $n$ with the maximum degree $\Delta$ such that $TW(T)>TW(T_1)$.
Moreover, any two starlike trees of order $n$ with the maximum degree $\Delta$ have the same terminal Wiener index, which is equal to $(n-1)(\Delta-1)$. Hence the proof is completed.\end{proof}

\end {document}